\documentclass[sn-mathphys]{sn-jnl}

\usepackage{amsthm}
\usepackage{amsmath}
\usepackage{natbib}
\usepackage{graphicx}
\usepackage{bbm}
\usepackage{amsfonts}
\usepackage{bm}
\usepackage{amssymb}

\jyear{2021}%

\theoremstyle{thmstyleone}%
\newtheorem{theorem}{Theorem}
\newtheorem{lemma}{Lemma}
\newtheorem{corollary}{Corollary}

\theoremstyle{thmstyletwo}%
\newtheorem{remark}{Remark}%

\theoremstyle{thmstylethree}%

\begin{document}

\title[Empirical priors in linear regression with unknown error variance]{High-dimensional properties for empirical priors in linear regression with unknown error variance}


\author*[1]{\fnm{Xiao} \sur{Fang}}\email{xiaofang@ufl.edu}

\author[2]{\fnm{Malay} \sur{Ghosh}}\email{ghoshm@ufl.edu}

\affil[]{\orgdiv{Department of Statistics}, \orgname{University of Florida}, \orgaddress{\street{  Griffin-Floyd Hall}, \city{Gainesville}, \postcode{32611}, \state{FL}, \country{USA}}}


\abstract{We study full Bayesian procedures for high-dimensional linear regression. We adopt data-dependent empirical priors introduced in  \citep{Martin:2017}. In their paper, these priors have nice posterior contraction properties and are easy to compute. Our paper extend their theoretical results to the case of  unknown error variance . Under proper sparsity assumption, we achieve model selection consistency, posterior contraction rates as well as Bernstein von-Mises theorem by analyzing multivariate t-distribution.}




\maketitle

\section{Introduction}
In a series of articles, Ryan Martin  and his colleagues introduced empirical priors in sparse high dimensional linear regression models (see for  example,   \citep{Martin:2014}, \citep{Martin:2017}, \citep{Martin:2020} and  \citep{Martin:2020.1}). These priors are all data dependent and achieve nice posterior contraction rates, specifically, concentration of the parameter of  interest around the true value at a very rapid rate. Moreover, these priors are quite satisfactory for both estimation and prediction as pointed out in their  articles.

While  \citep{Martin:2017} introduced priors both when the error variance $\sigma^2$ in a linear regression model is known and unknown, their theoretical results were proved only for known $\sigma^2$. The objective of this paper is to fill in the gap and obtain theoretical properties, namely posterior contraction rates for unknown $\sigma^2$.  The technical novelty of this approach is that  unlike the former, our algebraic manipulations require handling of  multivariate t-distributions rather than   multivariate normal distributions. In addition to the above, we have established model selection
consistency as well as a Bernstein von-Mises theorem in our proposed set up.

The outline of this paper is as follows. We have introduced the model and certain basic lemmas needed for the rest of the paper in Section 2. Posterior concentration as well as  model selection consistency results  are stated in Section 3. Bernstein von-Mises theorem and its application are stated in Section 4. All the proofs are given in Section 5. Some final remarks are made in Section 6.

\section{The Model}
Consider the standard linear regression model 
\begin{equation}
y=X\beta+ \sigma z,
\end{equation}
where $y$ is a $(n \times 1)$ vector of response variables, $X$ is $n \times p$ design matrix, $\beta$ is a $(p \times 1)$ regression parameter, $\sigma >0 $ is an unknown scale parameter and $z \sim N_n(0,I_n)$.

Let $S \subseteq \lbrace 1,2, \cdots, p\rbrace$, $\vert S\vert\le R$, $\vert S\vert$ denotes the cardinality of $S$. $R(\le n \ll p)$ is the rank of $X$. Also let $X_S$ $(n \times \vert S\vert)$ denote a submatrix of the column vectors of $X$ corresponding to the elements of S. It is assumed that $X_S^TX_S$ is nonsingular. The corresponding elements of the regression vector $\beta$ is denoted by $\beta_S$. Also, let $\hat{\beta}_S=(X_S^TX_S)^{-1}X_S^Ty$ denote the least square estimator of $\beta_S$.

\subsection{The prior}
The prior considered is as follows:\\
\\
(i)  $\beta_S\vert S,\sigma^2 \sim N_{\vert S\vert}(\hat{\beta}_{S},\sigma^2 {\gamma}^{-1}(X^T_SX_S)^{-1})$, $\beta_{S^c}$=0 with probability 1.\\
\\
(ii) $\sigma^2 \sim IG(a_0,b_0)$, i.e an inverse gamma distribution with shape and scale parameters  $a_0$ and $b_0$.\\
\\
(iii) Marginal priors for $S$: 
$ \pi(S) = {p \choose {s} }^{-1} f_n(s) $, where $s=\vert S\vert$, and $f_n(s) \propto (cp^a)^{-s}$
for  $s=0,1,\cdots,R$ 
and  $f_n(s)=0$ for $s=R+1,\cdots, p$.\\
\\
Then  the empirical joint prior for $(\beta,\sigma^2)$ is
\begin{equation}
\begin{split}
\Pi_n(\beta,\sigma^2) = &\sum_S \pi(S) N_{\vert S\vert}(\beta_{S}\vert \hat{\beta}_{S},\sigma^2 {\gamma}^{-1}(X^T_SX_S)^{-1}) \delta_{0}( \beta_{S^c}) \\
&\times \exp(-\frac{b_0}{\sigma^2})(\sigma^2)^{-a_0-1}b_0^{a_0}/\Gamma(a_0),
\end{split}
\end{equation}
where $\delta_0$ denotes the Dirac-Delta function.

\subsection{The posterior distribution}
Following \citep{Martin:2017}, we consider also a fractional likelihood 
\begin{displaymath}
L^{\alpha}(\beta, \sigma^2)=(2\pi\sigma^2)^{-\frac{n\alpha}{2}}\exp(-\frac{\alpha}{2\sigma^2} \|y-X\beta \|^2), 0< \alpha < 1.
\end{displaymath}
Then the posterior for $\beta$ conditional on $\sigma^2$ and $S$ is 

\begin{equation}
\begin{split}
&\pi^n(\beta\vert\sigma^2, S) \\ 
\propto & \exp[-\frac{\alpha}{2\sigma^2} \|y-X_S\beta_S-X_{S^c} \beta_{S^c}\|^2-\frac{\gamma}{2\sigma^2}(\beta_S-\hat{\beta}_S)^T X_S^T X_S(\beta_S-\hat{\beta}_S)] \\
 & \cdot (\sigma^2)^{-(\frac{n \alpha}{2}+a_0+1)}\delta_0(\beta_{S^c}).
\end{split}
\end{equation}
Consider the identity 
\begin{displaymath}
\begin{split}
& \alpha \|y-X_S\beta_S-X_{S^c} \beta_{S^c}\|^2+\gamma(\beta_S-\hat{\beta}_S)^T X_S^T X_S(\beta_S-\hat{\beta}_S)\\
= & (\alpha+\gamma)(\beta_S-\hat{\beta}_S)^T X_S^T X_S(\beta_S-\hat{\beta}_S)+\alpha \|y-\hat{y}_S\|^2+ \\
+ &\beta_{S^c}^T X_{S^c}^TX_{S^c} \beta_{S^c}-2\beta^T_{S^c}X_{S^c}^T(y-X_S \beta_S),
\end{split}
\end{displaymath}
where $\hat{y}_S=X_S \hat{\beta}_S$. Now recalling that $\beta_{S^c}$ is concentrated at $0$ with prior probability $1$, it follows that 
\begin{equation}
\begin{split}
\pi^n(\beta\vert\sigma^2, S)&=\pi^n(\beta_S\vert\sigma^2, S) \delta_0(\beta_{S^c}), \\
\pi^n(\beta_S\vert S,\sigma^2)&=N_{\vert S\vert}(\beta_S\vert\hat{\beta}_{S},\frac{\sigma^2}{\alpha+\gamma}(X^T_SX_S)^{-1}).
\end{split}
\end{equation}
Also the conditional posterior for $\sigma^2$ given $S$ is
\begin{equation}
\pi^n(\sigma^2\vert S) \sim IG(a_0+\frac{\alpha n}{2}, b_0 + \frac{\alpha}{2} \Vert y-\hat{y}_S \Vert ^2).
\end{equation}
\noindent
Thus the conditional distribution for $\beta_S$ given $S$ is 
\begin{equation}
\pi^n(\beta_S\vert S)=D^n_S(\beta_S) \frac{(\alpha+ \gamma)^{\vert S\vert/2} \{ b_0 + \frac{\alpha}{2} \Vert y-\hat{y}_S \Vert ^2\}^{-(a_0+\alpha n/{2})}}{\pi(S) \gamma ^{\vert S\vert/2}\Gamma(a_0+\alpha n/2)}, \label{beta_S given S}
\end{equation}
where 
\begin{equation}
D^n_S(\beta_S)=\frac{\pi(S)(2\pi)^{-\frac{\vert S\vert}{2}}\Gamma(\frac{n \alpha}{2} +\frac{\vert S\vert}{2}+a_0)\vert\gamma(X_S^T X_S)\vert^{\frac{1}{2}}}{(\frac{2b_0+(\beta_S- \hat{\beta}_S)^T(\gamma+\alpha)(X_S^T X_S)(\beta_S- \hat{\beta}_S)+\alpha \Vert y-\hat{y}_S \Vert^2}{2})^{\frac{n \alpha}{2} +\frac{\vert S\vert}{2}+a_0}}.
\end{equation}
\noindent
Finally, the  marginal posterior of  $S$  is
\begin{equation}
\pi^n(S) \propto \pi(S)(\frac{\gamma}{\alpha+\gamma})^{\vert S\vert/2} \{ b_0 + \frac{\alpha}{2} \Vert y-\hat{y}_S \Vert ^2\}^{-(a_0+\alpha n/{2})}.
\end{equation}
\begin{remark}
All  our results except Corollary \ref{uncertainty quantification} are obtained for all $0< \alpha < 1 $ and all $\gamma >0$. However, for higher-order  properties, such as credible probability of set, some conditions on $\alpha$ and $\gamma$ are required. For example, in Corollary \ref{uncertainty quantification} related to uncertainty, we assume $\alpha+ \gamma \le 1$. So we can always set 
 $\alpha$   close to $1$, $\gamma$   close to $0$, then the fractional likelihood is almost the same as the normal likelihood, and  $\beta_S\vert S,\sigma^2$ is like non-informative prior.
\end{remark}

\subsection{Tail bounds for the Chi squared distribution}

\begin{lemma}
For any $a>0$, we have
\begin{displaymath}
P(\vert\chi^2_p -p\vert>a) \le 2 \exp(-\frac{a^2}{4p}).
\end{displaymath} \label{lemma1}
\end{lemma}
\noindent
Proof: See Lemma 4.1 of \citep{Cao:2020}.

\begin{lemma} \noindent \\
(i) For any $c>0$, we have
\begin{displaymath}
P(\chi^2_p(\lambda)-(p+\lambda)>c) \le \exp(-\frac{p}{2} \{\frac{c}{p+\lambda}-log(1+\frac{c}{p+\lambda})\}).
\end{displaymath}
(ii) for $\omega <1$ then
\begin{displaymath}
P(\chi^2_p(\lambda) \le \omega \lambda) \le c_1 \lambda^{-1} \exp \{-\lambda(1-\omega)^2/8\},
\end{displaymath}
where $c_1>0$ is a constant.\\
\\
(iii)
For any $c>0$, $P(\chi^2_p(\lambda)-p \le -c) \le  \exp(-\frac{c^2}{4p})$.\label{lemma2}
\end{lemma} 
\noindent
Proof: See  \citep{Cao:2020} Lemma 4.2 for (i) and for (ii) on  \citep{Shin:2019}.  (iii)  follows from  the fact  if $U \sim \chi^2_p(\lambda)$, then $P(U>c)$ is strictly increasing in $\lambda$ for fixed $p$ and $c>0$. Hence
\begin{displaymath}
P(\chi^2_p(\lambda)-p \le -c) \le P(\chi^2_p-p \le -c) \le \exp(-\frac{c^2}{4p}).
\end{displaymath}

\subsection{Notations}
We define  $S_{\beta} =\{j:\beta_j \neq 0\}$.

Assume the true model is 
\begin{displaymath}
y \sim N_n(X \beta^{\star},\sigma_0^2 I_n) 
\end{displaymath}
and  let $S^{\star}=S_{\beta^{\star}}$, $s^{\star}=\vert S^{\star}\vert$.

We use $g_n \preceq M$ to denote $g_n \le M$ for sufficiently large $n$, $g_n \succeq M$ to denote $g_n \ge M$ for sufficiently large $n$.

\section{Posterior concentration rates}
Define empirical Bayes posterior probability of event $B \subset R^p$ as 
\begin{equation}
\Pi^n(B)=\frac{\int_{B} \int_{\sigma^2 > 0} L^{\alpha}(\beta,\sigma^2)\Pi_{n}(d \beta , d \sigma^2)}{\int_{R^p} \int_{\sigma^2 > 0} L^{\alpha}(\beta,\sigma^2)\Pi_{n}(d \beta , d \sigma^2)} =
\frac{\sum_S \int_{B}D^n_S(d\beta_S) \delta_{0^{S^c}}(d \beta_{S^c}) }{\sum_S \int_{R^p} D^n_S(d\beta_S)\delta_{0^{S^c}}(d \beta_{S^c})  }
\end{equation}
and  let $D_n= \sum_S {\int_{R^p} D^n_S(d\beta_S) \delta_{0^{S^c}}
(d \beta_{S^c})  }$. Then recalling the pdf of a multivariate $t$ distribution,
we have 
\begin{equation}
\begin{split}
 D_n 
=& \sum_S  \pi(S) \Gamma(\frac{n \alpha}{2} +a_0)(\frac{\gamma}{\gamma+\alpha})^{\frac{\vert S\vert}{2}}[\frac{1}{b_0+\frac{\alpha \Vert y-\hat{y}_S \Vert^2}{2}}]^{\frac{n \alpha}{2} +a_0} \\
\ge & \pi(S^{\ast}) \Gamma(\frac{n \alpha}{2} +a_0)(\frac{\gamma}{\gamma+\alpha})^{\frac{\vert S^{\ast}\vert}{2}}[\frac{1}{b_0+\frac{\alpha \Vert y-\hat{y}_{S^{\ast}} \Vert^2}{2}}]^{\frac{n \alpha}{2} +a_0}.
\end{split}
\end{equation}

Since $p \gg n$, we have to add some regularity conditions to get  posterior concentration results for our model.\\
\\
\textbf{regularity conditions:}\\
(A1) There exist constants $d_1$, $d_2 $, such that $0<d_1< \sigma_0^2 <d_2< \infty$,\\
(A2) $s^{\star} \log p = o(n)$.\\
(A3) $R\log p =o(n)$.\\
\\
\noindent
For a given $\Delta$, define $B_n(\Delta)=\{\beta \in R^p : \vert S_{\beta}\vert \ge \Delta\}$, i.e the set of $\beta$ vectors with no less then $\Delta$ non-zero entries. 

The following theorem implies that the posterior distribution is actually concentrated on a space of dimension  close to $s^{\star}$.
\begin{theorem}
Let $s^{\star} \le R$, and assume conditions (A1)-(A3) to hold. Then there exists constant $C > 1$, $G>0$ such that 
\begin{displaymath}
E_{\beta^{\star}}\{ \Pi^n(B_n(\Delta_n))\} \preceq \exp \{-Gs^{\star} \log(p/s^{\star})\}\to 0
\end{displaymath}
with $\Delta_n =Cs^{\star}$, uniformly in $\beta^{\star}$ as $n \to \infty$. \label{theorem1}
\end{theorem}

To get posterior concentration results, we first establish model selection result. The following theorem demonstrates that asymptotically our empirical Bayesian posterior  will not include any unnecessary variables.
\begin{theorem}
Let $s^{\star} \le R$, and assume conditions (A1)-(A3) to hold. Also, let the constant $a$  in the marginal prior of $S$ satisfy $p^a \gg s^{\star}$. Then $E_{\beta^{\star}}\{ \Pi^n(\beta:S_{\beta} \supset S^{\star} )\} \to 0$, uniformly in $\beta^{\star}$.\label{theorem2}
\end{theorem}

To get model selection consistency, it remains to show our empirical Bayesian posterior  will
asymptotically  not miss any true variables.

Define 
\begin{equation}
\kappa(s)=\kappa_X(s)= \inf_{\beta:0< \vert S_{\beta}\vert\le s} \frac{\parallel X \beta \parallel_2}{\parallel \beta \parallel_2}, \qquad s=1,2,\cdots p.
\end{equation}
Then $\kappa(s)$ is a non-increasing function of $s$. By  \citep{Martin:2017}, for any $\beta$,$\beta^ \prime$, $\kappa(\vert S_{\beta_S-\beta^\prime}\vert) \ge \kappa(\vert S_{\beta_S}\vert+\vert S_{\beta^\prime}\vert)$.

\begin{theorem}
Assume $s^{\star} \le R$, conditions (A1)-(A3) to hold and 
\begin{displaymath}
\min_{j\in S^{\star}}\vert\beta^{\star}_j\vert \ge \rho_n :=M^{\prime}(\log p)^{1/2},
\end{displaymath}
where $M^\prime = \frac{1}{\kappa(C^\prime s^{\star})}\{\frac{12 M d_2}{\alpha }\}^{1/2}$. We also assume $\kappa(C^{\prime}s^{\star})>0$, $C ^ \prime =C+1$ for $C$ as in Theorem \ref{theorem1} and $M$ is a constant with $M>1+a$,  $a>0$ being the parameter in the prior of $S$. Then
\begin{displaymath}
E_{\beta^{\star}} \{ \Pi^n(\beta : S_{\beta} \nsupseteq S^{\star}) \} \to 0.
\end{displaymath} \label{theorem3}
\end{theorem}

\begin{remark}
Intuitively, we are not able to distinguish between $0$ and a very small non-zero.  Hence, in Theorem \ref{theorem3} we define a cutoff $\rho_n$, which is similar to Theorem 5 in\citep{Martin:2017} and Theorem 5 in \citep{Castillo:2015}.
\end{remark}

\begin{corollary}
(Selection consistency)Assume that the conditions of Theorem \ref{theorem3}  hold. Then $E_{\beta^{\star}} \{\Pi^n(\beta : S_{\beta} = S^{\star})\} \to 1$.
\end{corollary}
\noindent
Proof: Follows immediately from Theorems \ref{theorem2} and \ref{theorem3}.

We now state our posterior concentration result. It is similar to the posterior concentration theorem in \citep{Martin:2017}. But our proof is completely different form theirs. They  apply Holder's inequality and Renyi divergence formula, while the key to our proof is  using the model selection consistency result.

Set
\begin{displaymath}
B_{\epsilon_n}=\{\beta \in R^p: \parallel X(\beta-\beta^{\star})\parallel_2^2 > \epsilon_n\},
\end{displaymath}
\begin{theorem}
Assume conditions (A1)-(A3) hold, then there exist  constant $M,G_2>0$ such that
\begin{displaymath}
E_{\beta^{\star}}\{\Pi^n(B_{M\epsilon_n})\} \preceq \exp({-G_2 \epsilon_n}) \to 0
\end{displaymath} 
uniformly in $\beta^{\star}$ as $n \to \infty$, where $\epsilon_n =s^{\star} \log (p/s^{\star})$. \label{theorem4}
\end{theorem}

By adding some conditions on $X$, we are able to seperate $\beta$ from $X \beta$ so we can get posterior consistency result for $\beta$.

Set
\begin{displaymath}
B^{\prime}_{\delta_n}=\{\beta \in R^p: \parallel \beta-\beta^{\star}\parallel_2^2 > \delta_n\},
\end{displaymath}
where $\delta_n$ is a positive sequence of constants to be specified.

\begin{theorem}
There exists a constant $M$ such that 
\begin{displaymath}
E_{\beta^{\star}}\{\Pi^n(B^{\prime}_{M\delta_n})\} \preceq \exp({-G \epsilon_n}) \to 0
\end{displaymath}
uniformly in $\beta^{\star}$ as $n \to \infty$, where 
\begin{displaymath}
\delta_n = \frac{s^{\star} \log (p/s^{\star})}{\kappa(C^{\prime }s^{\star})^2},
\end{displaymath}
$G>0$ is a constant, $C^{\prime}=C+1$, $C$ being the  constant in Theorem \ref{theorem1}.
\end{theorem}
\noindent
Proof: Same as the proof of Theorem 3 in \citep{Martin:2017}.\\
\\
\begin{remark} 
If  for any $\vert S\vert<R$ we have
\begin{displaymath}
 0<k_1<\lambda_{min}(X_S^T X_S/n)<\lambda_{max}(X_S^T X_S/n)<k_2<\infty,
\end{displaymath}
where $k_1$,$k_2$ are positive constant, then we get posterior consistency for the coefficient $\beta$ under $l_2$ norm .
\end{remark}

\section{Bernstein von-Mises Theorem}
In this section, we show that the posterior distribution of $\beta$ is asymptotically normal and this property  leads to many interesting results.
\begin{theorem}
(Bernstein von-Mises Theorem) Let $H$ denote Hellinger distance, $d_{TV}$ denote the total variation distance and $\Pi^n $ be empirical Bayes posterior distribution for $\beta$. If $s^{\star}=o(\sqrt{n}/\log n)$ and $E_{\beta^{\star}}\pi^n(S^{\star})\to 1$, then

\begin{equation}
E_{\beta^{\star}}H^2(\Pi^n,N_{s^{\star}}(\hat{\beta}_{S^{\star}},\frac{\sigma_0^2}{\alpha+\gamma}(X^T_{S^{\star}}X_{S^{\star}})^{-1})\bigotimes \delta_{0S^{\star c}}) \to 0
 \label{tv of posterior}
\end{equation}
uniformly in $\beta^{\star}$ as $n \to \infty$, where $\delta_{0S^{\star c}}$ denotes the point mass distribution for $\beta_{S^{\star c}}$ concentrated at the origin.

Since $d_{TV}(\cdot) \le H^2(\cdot)$, we also have
\begin{equation}
E_{\beta^{\star}}d_{TV}(\Pi^n,N_{s^{\star}}(\hat{\beta}_{S^{\star}},\frac{\sigma_0^2}{\alpha+\gamma}(X^T_{S^{\star}}X_{S^{\star}})^{-1})\bigotimes \delta_{0S^{\star c}}) \to 0 
 \label{tv of posterior}
\end{equation}
uniformly in $\beta^{\star}$ as $n \to \infty$.
\label{bernstein}
\end{theorem}

Define $H_n(t)=\Pi^n(\{\beta: x^T\beta \le t\})$, $t_{\gamma}=\inf\{t: H_n(t) \ge 1- \gamma\}$.

\begin{corollary} 
(valid uncertainty quantification) If $\alpha+ \gamma \le 1$, then Eq.(\ref{tv of posterior}) implies 
\begin{equation}
P_{\beta^{\star}}(t_{\gamma} \ge x^T {\beta}^{\star}) \ge 1- \gamma + o(1), \qquad n \to \infty, 
\end{equation}
which validates uncertainty quantification. \label{uncertainty quantification}
\end{corollary}

Consider a pair $(\tilde{X}, \tilde{y})$ where $\tilde{X} \in \mathbb{R}^{d \times p}$ is a given matrix of explanatory variable values at which we want to predict the corresponding response $\tilde{y} \in \mathbb{R}^d$. Let $f^n_{\tilde{X}}(\tilde{y}\vert S)$ be the conditional posterior predictive distribution of $\tilde{y}$, given $S$.  $f^n_{\tilde{X}}(\tilde{y}\vert S)$  is a t-distribution, with $2a+\alpha n$ degree of freedom, location $\tilde{X}_S \hat{\beta}_S$, and the scale matrix 
\begin{equation}
\frac{b+(\alpha /2) \|y-\hat y_S \|^2}{a+\alpha n /2}(I_d+\frac{1}{\alpha +\gamma }\tilde{X}_S(\tilde{X}_S^T\tilde{X}_S)^{-1}\tilde{X}_S^T) .
\end{equation}
Then  the  predictive distribution of $\tilde{y} $ is 
\begin{equation}
f^{n}_{\tilde{X}}(\tilde{y})=\sum_S \pi^n(S) f^n_{\tilde{X}}(\tilde{y}\vert S).
\end{equation}
For this predictive distribution, we have similar Bernstein von-Mises Theorem  as Theorem \ref{bernstein}. The proof is  similar.

\begin{theorem}
If $s^{\star}=o(\sqrt{n}/\log n)$ and $E_{\beta^{\star}}\pi^n(S^{\star})\to 1$, then
\begin{displaymath}
E_{\beta^{\star}}H^2(f^{n}_{\tilde{X}},N_d(\tilde{X}_{S^\star} \hat{\beta}_{S^\star},\sigma_0^2[I_d+\frac{1}{\alpha +\gamma }\tilde{X}_{S^\star}(\tilde{X}_{S^\star}^T\tilde{X}_{S^\star})^{-1}\tilde{X}_{S^\star}^T])) \to 0
\end{displaymath}
uniformly in $\beta^{\star}$ as $n \to \infty$.
 
Since $d_{TV}(\cdot) \le H^2(\cdot)$, we also have
\begin{displaymath}
E_{\beta^{\star}}d_{TV}(f^{n}_{\tilde{X}},N_d(\tilde{X}_{S^\star} \hat{\beta}_{S^\star},\sigma_0^2[I_d+\frac{1}{\alpha +\gamma }\tilde{X}_{S^\star}(\tilde{X}_{S^\star}^T\tilde{X}_{S^\star})^{-1}\tilde{X}_{S^\star}^T])) \to 0
\end{displaymath}
 uniformly in $\beta^{\star}$ as $n \to \infty$. \label{bernstein2}
\end{theorem}

\section{Proofs}
\textbf{Proof of Theorem \ref{theorem1}:}\\
Let $u_S=\vert S \cup S^{\star}\vert$, $R_{n,S}=\{y: y^T(I-P_{S \cup S^{\star}})y/\sigma^2_0 \ge \frac{\alpha+1}{2}(n-u_S)\}$.
\begin{equation}
\begin{split}
 &\Pi^n(B_n(\Delta_n))
 =  \frac{\sum_{\Delta_n \le \vert S\vert \le R }\int_{R^{\vert S\vert}}D^n_S(d\beta_S)}{D_n}    \\
 = & \frac{\sum_{\Delta_n \le \vert S\vert \le R}  \pi(S) \Gamma(\frac{n \alpha}{2} +a_0)(\frac{\gamma}{\gamma+\alpha})^{\frac{\vert S\vert}{2}}[\frac{1}{b_0+\frac{\alpha \Vert y-\hat{y}_S \Vert^2}{2}}]^{\frac{n \alpha}{2} +a_0}(\mathbbm{1}_{R_{n,S}}+\mathbbm{1}_{R^c_{n,S}})}{\sum_S  \pi(S) \Gamma(\frac{n \alpha}{2} +a_0)(\frac{\gamma}{\gamma+\alpha})^{\frac{\vert S\vert}{2}}[\frac{1}{b_0+\frac{\alpha \Vert y-\hat{y}_S \Vert^2}{2}}]^{\frac{n \alpha}{2} +a_0}}\\
 \le & \sum_{\Delta_n \le \vert S\vert \le R} \mathbbm{1}_{R^c_{n,S}} \\
 +& \sum_{\Delta_n \le \vert S\vert \le R} \frac{\pi(S)}{\pi(S^{\star})}(\frac{\gamma}{\alpha+\gamma})^{\vert S\vert/2-s^{\star}/2}[1+ \frac{\frac{\alpha}{2} y^T(P_{S}-P_{S^{\star}})y}{b_0+\frac{\alpha}{2} \parallel y-\hat{y}_{S}\parallel^2}]^{n\alpha/2+a_0}\cdot \mathbbm{1}_{R_{n,S}}\\
 \le & \sum_{\Delta_n \le \vert S\vert \le R} \mathbbm{1}_{R^c_{n,S}} \\
 +& \sum_{\Delta_n \le \vert S\vert \le R} \frac{\pi(S)}{\pi(S^{\star})}(\frac{\gamma}{\alpha+\gamma})^{\vert S\vert/2-s^{\star}/2}[1+ \frac{\frac{\alpha}{2} y^T(P_{S \cup S^{\star}}-P_{S^{\star}})y}{\frac{\alpha}{2} y^T(I- P_{S \cup S^{\star}})y}]^{n\alpha/2+a_0} \cdot \mathbbm{1}_{R_{n,S}}\\
 \le & \sum_{\Delta_n \le \vert S\vert \le R} \mathbbm{1}_{R^c_{n,S}} \\
 +& \sum_{\Delta_n \le \vert S\vert \le R} \frac{\pi(S)}{\pi(S^{\star})}(\frac{\gamma}{\alpha+\gamma})^{\vert S\vert/2-s^{\star}/2}[1+ \frac{ y^T(P_{S \cup S^{\star}}-P_{S^{\star}})y/\sigma_0^2}{\frac{\alpha+1}{2}(n-u_S)}]^{n\alpha/2+a_0} .
\end{split} \label{split B(delta)}
\end{equation}
Let $U_1$, $U_2$ denote the two terms on the right hand side of (\ref{split B(delta)}). It suffices to show  that there exists constant $C > 1$, $G>0$ such that  $E_{\beta^{\star}} U_i \preceq \exp \{-Gs^{\star} \log(p/s^{\star})\}$
 with $\Delta_n =Cs^{\star}$  uniformly in $\beta^{\star}$ for $i=1,2$.
 
Since
\begin{equation}
y^T(I-P_{S \cup S^{\star}})y/\sigma^2_0 \sim \chi^2_{n- u_S},
\end{equation}
by Lemma \ref{lemma1} we have 
\begin{equation}
\begin{split}
E_{\beta^{\star}} U_1 & = \sum_{\Delta_n \le \vert S\vert \le R} P(R^c_{n,S}) \le \sum_{\Delta_n \le \vert S\vert \le R}2 \exp \{-\frac{(1-\alpha)^2(n-u_S)}{16}\} \\
\preceq & \sum_{\Delta_n \le \vert S\vert \le R} 2 \exp \{-\frac{(1-\alpha)^2}{20}n\} \le R {p \choose R}2 \exp \{-\frac{(1-\alpha)^2}{20}n\}\\
\le & 2R p^R  \exp \{-\frac{(1-\alpha)^2}{20}n\} \preceq \exp \{-\frac{(1-\alpha)^2}{30}n\}.
 \end{split} \label{U_1 part}
\end{equation}

Let $Z_1 = y^T(P_{S \cup S^{\star}}-P_{S^{\star}})y/ \sigma^2_0$, then $Z_1 \sim \chi^2_{u_S-s^{\star}}$. By $1+x \le \exp(x)$,
\begin{equation}
\begin{split}
 [1+ \frac{ y^T(P_{S \cup S^{\star}}-P_{S^{\star}})y/\sigma_0^2}{\frac{\alpha+1}{2}(n-u_S)}]^{n\alpha/2+a_0} 
\le  \exp \{Z_1 \frac{(n\alpha/2+a_0)}{ \frac{\alpha+1}{2}(n-u_S)}\}
\preceq & \exp \{ \frac{\alpha Z_1}{1+\alpha}\}.
\end{split} 
\end{equation}
Using the mgf of a chi-squared distribution,
\begin{equation}
\begin{split}
& E_{\beta^{\star}} U_2  \\
 \le  &\sum_{\Delta_n \le \vert S\vert \le R} \frac{\pi(S)}{\pi(S^{\star})}(\frac{\gamma}{\alpha+\gamma})^{\vert S\vert/2-s^{\star}/2} E_{\beta^{\star}}\exp \{ \frac{\alpha Z_1}{1+\alpha}\}\\
 = & \sum_{\Delta_n \le \vert S\vert \le R} \frac{\pi(S)}{\pi(S^{\star})}(\frac{\gamma}{\alpha+\gamma})^{\vert S\vert/2-s^{\star}/2}(1-\frac{2 \alpha}{1+\alpha})^{-(\vert S \cup S^{\star} \vert-s^{\star})}\\
 \le & \sum_{\Delta_n \le \vert S\vert \le R} \frac{\pi(S)}{\pi(S^{\star})}(\frac{\gamma}{\alpha+\gamma})^{-s^{\star}/2}(\frac{1-\alpha}{1+\alpha})^{-\vert S\vert}\\
 = & \exp({cs^{\star}}) \frac{{p \choose s^{\star} }}{f_n(s^{\star})}\sum_{s= \Delta_n}^{R}\phi^sf_n(s),
 \end{split}
\end{equation}
where $\phi = (\frac{1-\alpha}{1+\alpha})^{-1} >1$. 

By \citep{Martin:2017} we have
$log\{\frac{{p \choose s^{\star} }}{f_n(s^{\star})}\sum_{s}\phi^sf_n(s)\} = O(s^{\star} log(p/s^{\star}))$. Also since $\phi >1$, $\sum_{s}\phi^sf_n(s)>1$. Hence
\begin{equation}
\begin{split}
&E_{\beta^{\star}} U_2 \\
\le &  \{\exp ({cs^{\star}}) \frac{{p \choose s^{\star} }}{f_n(s^{\star})}\sum_{s= \Delta_n}^{R}\phi^sf_n(s)\} \cdot \sum_{s}\phi^sf_n(s) \\
\le &\exp \{cs^{\star}+s^{\star} log(p/s^{\star})\}\sum_{s= \Delta_n}^{R}\phi^sf_n(s).
\end{split}
\end{equation}
From the expression of $f_n(s)$, we get 
\begin{equation}
\sum_{s= \Delta_n}^{R}\phi^sf_n(s)=O((\frac{\phi}{cp^a})^{\Delta_n+1}).
\end{equation}
and
\begin{displaymath}
(\frac{\phi}{cp^a})^{\Delta_n+1}=\exp \{-(\Delta_n+1)[a \log p + \log(c/\phi)]\}.
\end{displaymath}
So when $\Delta_n =Cs^{\star}$ , $C>a^{-1}$ and $G^{\prime} =\frac{aC-1}{2}$, 
\begin{equation}
E_{\beta^{\star}} U_2 \preceq \exp \{-G^{\prime} s^{\star} \log(p/s^{\star})\} \to 0 \label{U_2 part}
\end{equation} 
as $n \to \infty$.

By (\ref{U_1 part}) and (\ref{U_2 part}), when $\Delta_n =Cs^{\star}$ , $C>a^{-1}$ and $G =\frac{aC-1}{4}$, we have 
\begin{displaymath}
E_{\beta^{\star}}\{ \Pi^n(B_n(\Delta_n))\} \preceq \exp \{-Gs^{\star} \log(p/s^{\star})\}\to 0
\end{displaymath}
with $\Delta_n =Cs^{\star}$, uniformly in $\beta^{\star}$ as $n \to \infty$.\\
\\
\\
\textbf{Proof of Theorem \ref{theorem2}:} \\
Since 
\begin{displaymath}
E_{\beta^{\star}}\{ \Pi^n(\beta:S_{\beta} \supset S^{\star} )\}
=  \sum_{S:S \supset S^{\star}}E_{\beta^{\star}}\pi^n(S) ,
\end{displaymath}
by Theorem \ref{theorem1}, it suffices to show that
\begin{equation}
 \sum_{S:S \supset S^{\star},\vert S\vert \le Cs^{\star}}E_{\beta^{\star}}\pi^n(S) \to 0.
\end{equation}

Let  $\tilde{R}_{n,S}=\{y: y^T(I-P_{S})y/\sigma^2_0 \ge \frac{\alpha+1}{2}(n-s)\}$.
For $S \supset S^{\star}$, by $1+x \le \exp(x)$,
\begin{equation}
\begin{split}
\pi^n(S) \le & \mathbbm{1}_{\tilde{R}^c_{n,S}}+\frac{\pi^n(S)}{\pi^n(S^{\star})}\mathbbm{1}_{\tilde{R}_{n,S}}\\
\le & \mathbbm{1}_{\tilde{R}^c_{n,S}}+\frac{\pi(S)}{\pi(S^{\star})}(\frac{\gamma}{\alpha+\gamma})^{\vert S\vert/2-s^{\star}/2} \exp \{[y^T(P_{S }-P_{S^{\star}})y/\sigma_0^2] \frac{(n\alpha/2+a_0)}{ \frac{\alpha+1}{2}(n-s)}\}.
\end{split}
\end{equation}

Since $y^T(P_{S }-P_{S^{\star}})y/\sigma_0^2 \sim \chi^2_{s-s^{\star}}$ and $y^T(I-P_{S})y/\sigma^2_0 \sim \chi^2_{n-s}$, by Lemma \ref{lemma1} and mgf formula for chi-squared distribution, we have 
\begin{equation}
E_{\beta^{\star}}\pi^n(S) \le P(\tilde{R}^c_{n,S})+ \frac{\pi(S)}{\pi(S^{\star})}(\frac{\gamma}{\alpha+\gamma})^{\vert S\vert/2-s^{\star}/2} (\frac{1-\alpha}{1+\alpha})^{-(\vert S\vert-s^{\star})}.
\end{equation}
Let $z=(\frac{\gamma}{\alpha+\gamma})^{1/2}(\frac{1-\alpha}{1+\alpha})^{-1}$. Hence
\begin{equation}
\begin{split}
 & \sum_{S:S \supset S^{\star},\vert S\vert \le Cs^{\star}}E_{\beta^{\star}}\pi^n(S) \le \sum_{S:S \supset S^{\star},\vert S\vert \le Cs^{\star}}P(\tilde{R}^c_{n,S})+\sum_{S:S \supset S^{\star},\vert S\vert \le Cs^{\star}}\frac{\pi(S)}{\pi(S^{\star})} z^{\vert S\vert-s^{\star}}\\
&\preceq \exp \{-\frac{(1-\alpha)^2}{30}n \} + \sum_{s^{\star}<s \le Cs^{\star}} \frac{{p-s^{\star} \choose p-s}{p \choose s^{\star}}}{{p \choose s}}(\frac{z}{cp^a})^{s-s^{\star}} \\
& = \exp \{-\frac{(1-\alpha)^2}{30}n \}+\sum_{s^{\star}<s \le Cs^{\star}} {s \choose s-s^{\star}}(\frac{z}{cp^a})^{s-s^{\star}}\\
&\le  \exp \{-\frac{(1-\alpha)^2}{30}n\}+\sum_{s^{\star}<s \le Cs^{\star}} s^{s-s^{\star}}(\frac{z}{cp^a})^{s-s^{\star}} \\
&  \le \exp \{-\frac{(1-\alpha)^2}{30}n\}+\sum_{s^{\star}<s \le Cs^{\star}} (\frac{Czs^{\star}}{cp^a})^{s-s^{\star}} \\
&\le \exp \{-\frac{(1-\alpha)^2}{30}n\}+\frac{Czs^{\star}}{cp^a} \times O(1).
\end{split}
\end{equation}
Then RHS  goes to 0, when $p^a  \gg s^{\star}$.\\
\\
\\
\textbf{Proof of Theorem \ref{theorem3}:}\\
Since
\begin{displaymath}
E_{\beta^{\star}} \{\Pi^n(\beta : S_{\beta} \nsupseteq S^{\star})\} = \sum_{S:S \nsupseteq S^{\star}}E_{\beta^{\star}}\pi^n(S),
\end{displaymath}
by Theorem \ref{theorem1}, it suffices to show 
\begin{equation}
\sum_{S:S \nsupseteq S^{\star},\vert S\vert \le Cs^{\star}}E_{\beta^{\star}}\pi^n(S) \to 0 . \label{sum formula}
\end{equation}

Define 
\begin{equation}
\begin{split}
&L_{n,S} \\
= & \{y: \frac{(\alpha+1)(n-\vert S\vert)}{2}\le \parallel y-\hat{y}_{S}\parallel^2 / {\sigma_0}^2 \le 2(n-\vert S\vert)+ 2\parallel (I-P_S)X\beta^{\star}\parallel^2  /{\sigma_0}^2 \},
\end{split}
\end{equation}

\begin{equation}
\begin{split}
\pi^n(S) & \le  \mathbbm{1}_{L^c_{n,S}}+\frac{\pi^n(S)}{\pi^n(S^{\star})}\mathbbm{1}_{L_{n,S}}\\
& = \mathbbm{1}_{L^c_{n,S}}+\frac{\pi(S)}{\pi(S^{\star})}(\frac{\gamma}{\alpha+\gamma})^{\vert S\vert/2-s^{\star}/2}[\frac{b_0+\frac{\alpha}{2} \parallel y-\hat{y}_{S^{\star}}\parallel^2}{b_0+\frac{\alpha}{2} \parallel y-\hat{y}_{S}\parallel^2}]^{n\alpha/2+a_0}\mathbbm{1}_{L_{n,S}}. \label{pi_n(S) general inequality}
\end{split}
\end{equation}

By $1+x \le \exp(x)$,
\begin{equation}
\begin{split}
&[\frac{b_0+\frac{\alpha}{2} \parallel y-\hat{y}_{S^{\star}}\parallel^2}{b_0+\frac{\alpha}{2} \parallel y-\hat{y}_{S}\parallel^2}]^{n\alpha/2+a_0}\mathbbm{1}_{L_{n,S}}\\
= & [1+\frac{\frac{\alpha}{2} y^T(P_S-P_{S^{\star}})y}{b_0+\frac{\alpha}{2} \parallel y-\hat{y}_{S}\parallel^2}]^{n\alpha/2+a_0}\mathbbm{1}_{L_{n,S}}\\
\le & \exp\{\frac{\alpha}{2} y^T(P_S-P_{S^{\star}})y \frac{n\alpha/2+a_0}{b_0+\frac{\alpha}{2} \parallel y-\hat{y}_{S}\parallel^2}\}\mathbbm{1}_{L_{n,S}}\\
\le & \exp\{ y^T(P_S-P_{S^{\star}})y \frac{\alpha(n\alpha/2+a_0)}{2b_0+\alpha(\alpha+1)\sigma_0^2 (n-\vert S\vert)}\} \\
&+\exp\{ y^T(P_S-P_{S^{\star}})y \cdot \frac{\alpha(n\alpha/2+a_0)}{2b_0+2\alpha{\sigma_0}^2[2(n-\vert S\vert)+ 2\parallel (I-P_S)X\beta^{\star}\parallel^2  /{\sigma_0}^2]}\} .\label{ratio control}
\end{split}
\end{equation}

Since 
\begin{displaymath}
\parallel y-\hat{y}_{S}\parallel^2 / {\sigma_0}^2 \sim \chi^2_{n-s} (\parallel (I-P_S)X\beta^{\star}\parallel^2  /{\sigma_0}^2), 
\end{displaymath}
\noindent
\\
by Lemma \ref{lemma2} we have 
\begin{equation}
\begin{split}
& \sum_{\vert S\vert \le Cs^{\star}} P(L^c_{n,S}) \le \sum_{\vert S\vert \le Cs^{\star}}\{\exp \{-\frac{n-\vert S\vert}{10}\}+2 \exp [-\frac{(1-\alpha)^2(n-\vert S\vert)}{16}]\}\\
&\preceq \exp \{-\frac{(1-\alpha)^2 n}{20}\} \sum_{ s \le Cs^{\star}}{p \choose s} \le \exp \{-\frac{(1-\alpha)^2 n}{20}\}Cs^{\star} p^{Cs^{\star}} \preceq \exp\{-\frac{(1-\alpha)^2 n}{30}\} \to 0 ,
\end{split} \label{Q_n complement}
\end{equation}
as $n \to \infty$.

Plugging $y=X \beta^{\star}+\sigma_0 \epsilon$ into $y^T(P_S-P_{S^{\star}})y$, where $\epsilon \sim N_n(0,I)$, we get 
\begin{displaymath}
-\parallel(I-P_S)X\beta^{\star}\parallel^2-2\sigma_0 \epsilon^T(I-P_S)X\beta^{\star}+\sigma^2_0 \epsilon^T(P_S-P_{ S^{\star}}) \epsilon.
\end{displaymath}
Bound the right-most quadratic form above as follows,
\begin{displaymath}
\epsilon^T(P_S-P_{S^{\star}}) \epsilon = \epsilon^T(P_S-P_{S \cap S^{\star}}) \epsilon-\epsilon^T(P_{S^{\star}}-P_{S \cap S^{\star}}) \epsilon \le \epsilon^T(P_S-P_{S \cap S^{\star}}) \epsilon,
\end{displaymath} 
so
\begin{displaymath}
y^T(P_S-P_{S^{\star}})y \le -\parallel(I-P_S)X\beta^{\star}\parallel^2-2\sigma_0 \epsilon^T(I-P_S)X\beta^{\star}+\sigma^2_0 \epsilon^T(P_S-P_{S \cap S^{\star}}) \epsilon.
\end{displaymath}
We also observe that $(I-P_S)(P_S-P_{S \cap S^{\star}})=0$, which implies that 
\begin{displaymath}
\epsilon^T(I-P_S)X\beta^{\star} \perp \epsilon^T(P_S-P_{S \cap S^{\star}}) \epsilon.
\end{displaymath}
Then by the mgf of normal and chi-squared distribution, we have
\begin{equation}
E_{\beta^{\star}}\exp \{\frac{l}{2\sigma_0^2} \{y^T(P_S-P_{S^{\star}})y\}\} \le (1-l)^{-\frac{1}{2}(\vert S\vert-\vert S \cap S^{\star}\vert)} \exp\{-\frac{l(1-l)}{2 \sigma_0^2} \parallel (I-P_S)X \beta ^{\star} \parallel^2 \}.
\end{equation}
Hence,
\begin{equation}
\begin{split}
&E_{\beta^{\star}} \{[\frac{b_0+\frac{\alpha}{2} \parallel y-\hat{y}_{S^{\star}}\parallel^2}{b_0+\frac{\alpha}{2} \parallel y-\hat{y}_{S}\parallel^2}]^{n\alpha/2+a_0}\mathbbm{1}_{L_{n,S}}\}\\
\le & E_{\beta^{\star}} \exp\{ y^T(P_S-P_{S^{\star}})y \frac{\alpha(n\alpha/2+a_0)}{2b_0+\alpha(\alpha+1)\sigma_0^2 (n-\vert S\vert)}\} \\
&+E_{\beta^{\star}} \exp\{ y^T(P_S-P_{S^{\star}})y \cdot \frac{\alpha(n\alpha/2+a_0)}{2b_0+2\alpha{\sigma_0}^2[2(n-\vert S\vert)+ 2\parallel (I-P_S)X\beta^{\star}\parallel^2  /{\sigma_0}^2}\} \\
\le &  \sum_{i=1}^{2}(1-\alpha_i)^{-\frac{1}{2}(\vert S\vert-\vert S \cap S^{\star}\vert)}\exp\{-\frac{\alpha_i(1-\alpha_i)}{2 \sigma_0^2}\parallel (I-P_S)X\beta^{\star}\parallel^2\},
\end{split}\label{Qn part control}
\end{equation}
\noindent
where $\alpha_1=\frac{2 \sigma_0^2 \alpha(n\alpha/2+a_0)}{2b_0+\alpha(\alpha+1)\sigma_0^2 (n-\vert S\vert)} \le \frac{\alpha}{1+\alpha}$, $\alpha_2=\frac{2 \sigma_0^2 \alpha(n\alpha/2+a_0)}{2b_0+2\alpha{\sigma_0}^2[2(n-\vert S\vert)+ 2\parallel (I-P_S)X\beta^{\star}\parallel^2  /{\sigma_0}^2]} \le \frac{\alpha}{4}$.

Since 
\begin{displaymath}
\parallel (I-P_S)X\beta^{\star}\parallel^2=\parallel (I-P_S)X_{S\cap S^{\star}}\beta_{S\cap S^{\star}}^{\star}\parallel^2,
\end{displaymath}
it follows from Lemma 5 of  \citep{Arias-Castro:2014} that 
\begin{displaymath}
\parallel (I-P_S)X\beta^{\star}\parallel^2 \ge \kappa(S\cup S^{\star})^2(\vert S^{\star}\vert-\vert S\cap S^{\star}\vert)\rho_n^2.
\end{displaymath}
Next, we have $\vert S\cup S^{\star}\vert\le \vert S\vert+ \vert S^{\star}\vert\le C^\prime \vert S^{\star}\vert$. By the monotonicity of $\kappa$  we also have 
\begin{equation}
\parallel (I-P_S)X\beta^{\star}\parallel^2 \ge \kappa(C^\prime s^{\star})^2(\vert S^{\star}\vert-\vert S\cap S^{\star}\vert)\rho_n^2 \ge (M^\prime)^2\kappa(C^\prime s^{\star})^2(\vert S^{\star}\vert-\vert S\cap S^{\star}\vert)\log p
\end{equation}
Then 
\begin{displaymath}
\begin{split}
\frac{\alpha_1(1-\alpha_1)}{2 \sigma_0^2}\parallel (I-P_S)X\beta^{\star}\parallel^2  \ge & \frac{\alpha}{2(\alpha+1)^2 d_2}(M^\prime)^2\kappa(C^\prime s^{\star})^2(\vert S^{\star}\vert-\vert S\cap S^{\star}\vert)\log p \\
> & M(\vert S^{\star}\vert-\vert S\cap S^{\star}\vert)\log p,
\end{split}
\end{displaymath}

\begin{displaymath}
\begin{split}
&\frac{\alpha_2(1-\alpha_2)}{2 \sigma_0^2}\parallel (I-P_S)X\beta^{\star}\parallel^2  
\ge \frac{3\alpha_2}{8 \sigma_0^2}\parallel (I-P_S)X\beta^{\star}\parallel^2 \\
\ge &\frac{3\sigma_0^2 \alpha(n\alpha/2+a_0)(M^\prime)^2\kappa(C^\prime s^{\star})^2(\vert S^{\star}\vert-\vert S\cap S^{\star}\vert)\log p}{8b_0+8\alpha{\sigma_0}^2[2(n-\vert S\vert)+ 2(M^\prime)^2\kappa(C^\prime s^{\star})^2(\vert S^{\star}\vert-\vert S\cap S^{\star}\vert)\log p  /{\sigma_0}^2]} \\
\ge &  \frac{(3n \alpha^2 /2)(M^\prime)^2\kappa(C^\prime s^{\star})^2(\vert S^{\star}\vert-\vert S\cap S^{\star}\vert)\log p}{18\alpha{\sigma_0}^2[n+ s^{\star}\log p \cdot (24 d_2 M)/(d_1 \alpha)]} \\
\ge & \frac{\alpha}{12  d_2} (M^\prime)^2\kappa(C^\prime s^{\star})^2(\vert S^{\star}\vert-\vert S\cap S^{\star}\vert)\log p = M(\vert S^{\star}\vert-\vert S\cap S^{\star}\vert)\log p.
\end{split}
\end{displaymath}

By(\ref{pi_n(S) general inequality})(\ref{ratio control})(\ref{Q_n complement})(\ref{Qn part control}), we have
\begin{equation}
\begin{split}
&\sum_{S:S \nsupseteq S^{\star},\vert S\vert \le Cs^{\star}}E_{\beta^{\star}}\pi^n(S) \\
\le&  \sum_{S:S \nsupseteq S^{\star},\vert S\vert \le Cs^{\star}}  E_{\beta^{\star}}\{\mathbbm{1}_{L^c_{n,S}}+\frac{\pi(S)}{\pi(S^{\star})}(\frac{\gamma}{\alpha+\gamma})^{\vert S\vert/2-s^{\star}/2}[\frac{b_0+\frac{\alpha}{2} \parallel y-\hat{y}_{S^{\star}}\parallel^2}{b_0+\frac{\alpha}{2} \parallel y-\hat{y}_{S}\parallel^2}]^{n\alpha/2+a_0}\mathbbm{1}_{L_{n,S}}\}\\
\preceq & \exp \{-\frac{(1-\alpha)^2 n}{30}\}+\sum_{S:S \nsupseteq S^{\star},\vert S\vert \le Cs^{\star}}2\frac{\pi(S)}{\pi(S^{\star})}\nu^{s^{\star}-\vert S\vert}({\sqrt{2}}p^{-M})^{\vert S^{\star}\vert-\vert S\cap S^{\star}\vert},
\end{split}
\end{equation}
where $\nu = (\frac{2\gamma}{\alpha+\gamma})^{-1/2}$.

Plug in the  prior of $S$,  and let $t$ be the number of variables in   $S \cap S^\star$. We get
\begin{displaymath}
\begin{split}
&\sum_{S:S \nsupseteq S^{\star},\vert S\vert \le Cs^{\star}} \frac{\pi(S)}{\pi(S^{\star})}\nu^{s^{\star}-\vert S\vert}({\sqrt{2}}p^{-M})^{\vert S^{\star}\vert-\vert S\cap S^{\star}\vert} \\
\le & \sum_{s=0}^{C s^{\star}} \sum_{t=1} ^{\min \{s, s^\star\}} \frac{{s^\star \choose t}{{p-s^\star} \choose {s-t} }{p \choose s^\star}}{{p \choose s}}(\nu c p^a)^{s^\star-s}(\sqrt{2}p^{-M})^{s^\star-t}\\
= & \sum_{s=0}^{C s^{\star}} \sum_{t=1} ^{\min \{s, s^\star\}} {S \choose t}{{p-s} \choose {s^\star -t}}(\nu c p^a)^{s^\star-s}(\sqrt{2}p^{-M})^{s^\star-t}\\
\leq & \sum_{s=0}^{C s^{\star}} \sum_{t=1} ^{\min \{s, s^\star\}}  s^{s-t} p^{s^\star -t}(\nu c p^a)^{s^\star-s}(\sqrt{2}p^{-M})^{s^\star-t} \\
= & \sum_{s=0}^{s^\star -1} \sum_{t=1} ^{s}  s^{s-t} p^{s^\star -t}(\nu c p^a)^{s^\star-s}(\sqrt{2}p^{-M})^{s^\star-t} \\
&+ \sum_{s=s ^ \star}^{C s^{\star}} \sum_{t=1} ^{ s^\star}  s^{s-t} p^{s^\star -t}(\nu c p^a)^{s^\star-s}(\sqrt{2}p^{-M})^{s^\star-t} \\
= & \sum_{s=0}^{s^\star -1} \sum_{t=1} ^{s}  (\nu c p^a /s)^{s^\star-s}(\sqrt{2}s p^{1-M})^{s^\star-t} 
+ \sum_{s=s ^ \star}^{C s^{\star}} \sum_{t=1} ^{ s^\star}  (\nu c p^a/s)^{s^\star-s}(\sqrt{2}sp^{1-M})^{s^\star-t} \\
\preceq & \sum_{s=0}^{s^\star -1} (\nu c p^{1+a-M} /s)^{s^\star-s}
+ s^ \star p^{1-M} \sum_{s=s ^ \star}^{C s^{\star}}   (\nu c p^a/s)^{s^\star-s} \to 0,
\end{split}
\end{displaymath}
which implies $\sum_{S:S \nsupseteq S^{\star},\vert S\vert \le Cs^{\star}}E_{\beta^{\star}}\pi^n(S) \to 0 $.

As we  see, $\parallel (I-P_S)X\beta^{\star}\parallel^2$ plays an important role in the proof of Theorem \ref{theorem3}. We modify this proof to get a useful result.
\begin{lemma}
Assume conditions (A1)-(A3) hold, define 
\begin{displaymath}
H_{S}^n=\{\beta^{\star} : \parallel (I-P_S)X\beta^{\star}\parallel^2 \ge Ks^{\star} \log (p/s^\star)\}, 
\end{displaymath}
$K>(2a+3)\frac{12 d_2}{\alpha }$ is a constant, with $a>0$ the constant in marginal prior of $S$, $C$  is  a constant as in Theorem \ref{theorem1}. Then there exists constant $G_1 >0$, such that
\begin{displaymath}
\sum_{S:S \nsupseteq S^{\star},\vert S\vert \le Cs^{\star}}\mathbbm{1}_{H_{S}^n}\cdot E_{\beta^{\star}} \pi^n(S) \preceq \exp \{-G_1 s^{\star} \log (p/s^{\star})\} \to 0
\end{displaymath}
uniformly in $\beta^{\star}$ as $n \to \infty$.\label{lemma3}
\end{lemma}

\begin{proof}
By (\ref{pi_n(S) general inequality}), (\ref{Q_n complement}), (\ref{Qn part control}), for  $S \in \{S:S \nsupseteq S^{\star},\vert S\vert \le Cs^{\star}\}$,
\begin{equation}
\begin{split}
& \mathbbm{1}_{H_{S}^n} \cdot E_{\beta^{\star}}\pi^n(S) \le \mathbbm{1}_{H_{S}^n}\cdot E_{\beta^{\star}}[\pi^n(S) \mathbbm{1}_{Q_n}] +P(Q_n^c) \\
\le & \mathbbm{1}_{H_{S}^n} \cdot \frac{\pi(S)}{\pi(S^{\star})}(\frac{\gamma}{\alpha+\gamma})^{\frac{\vert S\vert}{2}-\frac{s^{\star}}{2}}\sum_{i=1}^{2}(1-\alpha_i)^{-\frac{1}{2}(\vert S\vert-\vert S \cap S^{\star}\vert)}\exp\{-\frac{\alpha_i(1-\alpha_i)}{2 \sigma_0^2}\parallel (I-P_S)X\beta^{\star}\parallel^2\} \\
&+ \exp \{-\frac{(1-\alpha)^2 n}{20}\}.
\end{split}
\end{equation}
\noindent
Then on $H_{S}^n$,
\begin{displaymath}
\begin{split}
\frac{\alpha_1(1-\alpha_1)}{2 \sigma_0^2}\parallel (I-P_S)X\beta^{\star}\parallel^2  \ge \frac{K\alpha}{2(\alpha+1)^2 d_2} s^{\star} \log (p/s^{\star}),
\end{split}
\end{displaymath}
\noindent
\begin{displaymath}
\begin{split}
\frac{\alpha_2(1-\alpha_2)}{2 \sigma_0^2}\parallel (I-P_S)X\beta^{\star}\parallel^2  \ge \frac{K\alpha}{12  d_2} s^{\star} \log (p/s^{\star}).
\end{split}
\end{displaymath}
So 
\begin{equation}
\mathbbm{1}_{H_{S}^n} \cdot E_{\beta^{\star}} [\pi^n(S)\mathbbm{1}_{Q_n}] \le
2\frac{\pi(S)}{\pi(S^{\star})}\nu^{s^{\star}-\vert S\vert}[{\sqrt{2}}(p/s^{\star})^{-M_0}]^{s^{\star}},
\end{equation}
where $\nu = (\frac{2\gamma}{\alpha+\gamma})^{-1/2}$, $M_0=\frac{K\alpha}{12  d_2}$.
\end{proof}

Set
\begin{equation}
B_{\epsilon_n}=\{\beta \in \mathbb{R}^p: \parallel X(\beta -\beta^{\star})\parallel_2^2 \ge\epsilon_n\},
\end{equation}
where $\epsilon_n$ is a positive sequence of constants to be specified later.\\
\\
\\
\textbf{Proof of Theorem \ref{theorem4}:}\\
Since
\begin{equation}
\Pi^n(B_{M\epsilon_n}) = \sum_{S}\pi^n(B_{M\epsilon_n}\vert S) \pi^n(S) \le \sum_{\vert S\vert \le Cs^{\star}}\pi^n(B_{M\epsilon_n}\vert S) \pi^n(S) + \sum_{\vert S\vert > Cs^{\star}} \pi^n(S),
\end{equation}
by Theorem \ref{theorem1} it suffices to show that
\begin{equation}
E_{\beta^{\star}} \sum_{\vert S\vert \le Cs^{\star}}\pi^n(B_{M\epsilon_n}\vert S) \pi^n(S) \preceq \exp ({-G_2 \epsilon_n} ).
\end{equation}

Let $\beta_{S+}$ be a $p$-vector  by augmenting $\beta_S$ with $\beta_j=0$ for all $j \in S^c$, and  $B_{M\epsilon_n}(S)$ be the set of all $\beta_S$ such that $\beta_{S+} \in B_n$, $\pi^n(B_{M\epsilon_n}\vert S)=\pi^n(B_{M\epsilon_n}(S)\vert S)$.\\
\\


Define 
\begin{displaymath}
Q_{n,S}=\{\beta^{\star} \in \mathbb{R}^p: \parallel(I-P_S)X\beta^{\star}\parallel_2^2 \ge M\epsilon_n /2\} .
\end{displaymath}
Then
\begin{equation}
\begin{split}
& \sum_{\vert S\vert \le Cs^{\star}}\pi^n(B_{M\epsilon_n}(S)\vert S) \pi^n(S) \\
=&  \sum_{S^{\star} \nsubseteq S, \vert S\vert \le Cs^{\star}}\pi^n(B_{M\epsilon_n}(S)\vert S)\pi^n(S)+ \sum_{S^{\star} \subseteq S, \vert S\vert \le Cs^{\star}}\pi^n(B_{M\epsilon_n}(S)\vert S)\pi^n(S)\\
& \le W_1+W_2+ W_3, 
\end{split}
\end{equation}
\noindent
where  $W_1= \sum_{S^{\star} \nsubseteq S, \vert S\vert \le Cs^{\star}}\pi^n(S)\mathbbm{1}_{Q_{n,S}}
$,  $W_2= \sum_{S^{\star} \nsubseteq S, \vert S\vert \le Cs^{\star}}\pi^n(B_{M\epsilon_n}(S)\vert S)\mathbbm{1}_{Q^c_{n,S}}$, and $W_3= \sum_{S^{\star} \subseteq S, \vert S\vert \le Cs^{\star}}\pi^n(B_{M\epsilon_n}(S)\vert S)$.
Thus it suffices to show that $E_{\beta^{\star}}W_i \preceq \exp({-G \epsilon_n})$, $i=1,2,3$ for some $G>0$.

For $W_1$, let $M>(2a+3) \frac{24 d_2}{\alpha}$. Then by Lemma \ref{lemma3}, we have $E_{\beta^{\star}}W_1 \preceq \exp{(- G_1 \epsilon_n)}$.

Next we  consider $W_2$. By (\ref{beta_S given S}) and generalized Holder's inequality, we have 
\begin{equation}
\begin{split}
& E_{\beta^{\star}}\pi^n(B_{M\epsilon_n}(S)\vert S)\\
=&(\frac{\alpha+ \gamma}{2 \pi})^{\vert S\vert/2}\vert X_S^TX_S\vert^{1/2}\frac{\Gamma(a_0+\alpha n/2+\vert S\vert/2)}{\Gamma(a_0+\alpha n/2)}\\
& \times \int_{B_{M\epsilon_n}(S)}\frac{\{ b_0 + \frac{\alpha}{2} \Vert y-\hat{y}_S \Vert ^2\}^{(a_0+\alpha n/{2})}}{\{ b_0 + \frac{\alpha}{2} \Vert y-\hat{y}_S \Vert ^2+ \frac{\gamma+\alpha}{2} \parallel X_S(\beta_S- \hat{\beta}_S)\parallel ^2 \}^{(a_0+\alpha n/{2}+\vert S\vert/2)}} d\beta_S \\
\le &(\frac{\alpha+ \gamma}{2 \pi})^{\vert S\vert/2}\vert X_S^TX_S\vert^{1/2}\frac{\Gamma(a_0+\alpha n/2+\vert S\vert/2)}{\Gamma(a_0+\alpha n/2)}\\
& \times \int_{B_{M\epsilon_n}(S)} [E_{\beta^{\star}} \frac{\{ b_0 + \frac{\alpha}{2} \Vert y-\hat{y}_S \Vert ^2\}^{(a_0+\alpha n/{2}-1/2)\cdot 2}}{\{ b_0 + \frac{\alpha}{2} \Vert y-\hat{y}_S \Vert ^2+ \frac{\gamma+\alpha}{2} \parallel X_S(\beta_S- \hat{\beta}_S)\parallel ^2 \}^{(a_0+\alpha n/{2}-1/2)\cdot 2}}]^{1/2} \\
& \cdot [E_{\beta^{\star}} \frac{1}{\{ b_0 + \frac{\alpha}{2} \Vert y-\hat{y}_S \Vert ^2+ \frac{\gamma+\alpha}{2} \parallel X_S(\beta_S- \hat{\beta}_S)\parallel ^2 \}^{(\vert S\vert/2+1/2)\cdot 4}}]^{1/4} \\
& \cdot [E_{\beta^{\star}}(b_0 + \frac{\alpha}{2} \Vert y-\hat{y}_S \Vert ^2)^{4/2}]^{1/4}\cdot d\beta_S.
\end{split}\label{W_2 S part}
\end{equation}

We now observe that 
\begin{displaymath}
\Vert y-\hat{y}_S \Vert ^2 \sim \sigma_0^2 \chi^2_{n-\vert S\vert}(\frac{\parallel(I-P_S)X\beta^{\star}\parallel_2^2}{\sigma_0^2})
\end{displaymath}
and 
\begin{displaymath}
\parallel X_S(\beta_S- \hat{\beta}_S)\parallel ^2 \vert \beta_S \sim \sigma_0^2 \chi^2_{\vert S\vert}(\frac{\parallel P_S X(\beta_{S+} -\beta^{\star})\parallel_2^2}{\sigma_0^2})\quad .
\end{displaymath}
Also
\begin{equation}
\begin{split}
\parallel X(\beta_{S+} -\beta^{\star})\parallel_2^2= &\parallel (I-P_S)X(\beta_{S+} -\beta^{\star})\parallel_2^2+\parallel P_S X(\beta_{S+} -\beta^{\star})\parallel_2^2\\
= & \parallel (I-P_S)X\beta^{\star}\parallel_2^2+\parallel P_S X(\beta_{S+} -\beta^{\star})\parallel_2^2 \quad.
\end{split}
\end{equation}
Hence by Lemmas \ref{lemma1} and  \ref{lemma2}, for ${\vert S\vert \le Cs^{\star}}$, on $Q^c_{n,S} \cap B_{M\epsilon_n}(S)$, we have 
\begin{equation}
\parallel P_S X(\beta_{S+} -\beta^{\star})\parallel_2^2 \ge M\epsilon_n /2  ,\quad
P(\Vert y-\hat{y}_S \Vert ^2 > 2\sigma_0^2 n) \le \exp(-n/10) \quad ,
\end{equation}
and
\begin{equation}
\begin{split}
& P(\parallel X_S(\beta_S- \hat{\beta}_S)\parallel ^2 \le  M \epsilon_n /4\vert \beta_S)\\
\le &
P(\parallel X_S(\beta_S- \hat{\beta}_S)\parallel ^2 \le \parallel P_S X(\beta_{S+} -\beta^{\star})\parallel_2^2/2\vert \beta_S)\\
\le & \frac{C^\prime \sigma_0^2}{\parallel P_S X(\beta_{S+} -\beta^{\star})\parallel_2^2} 
\exp\{-\frac{\parallel P_S X(\beta_{S+} -\beta^{\star})\parallel_2^2}{32\sigma_0^2}\}\\
\le &\frac{2 c_2 C^\prime}{M \epsilon_n} \exp\{-\frac{\parallel P_S X(\beta_{S+} -\beta^{\star})\parallel_2^2}{32\sigma_0^2}\} \le \frac{2 c_2 C^\prime}{M \epsilon_n} \exp\{-\frac{M \epsilon_n}{64\sigma_0^2}\},
\end{split}
\end{equation}
where $C^{\prime}$ is a constant.

Let
\begin{equation}
\begin{split}
&A=\{ y:\Vert y-\hat{y}_S \Vert ^2 \ge 2\sigma_0^2 n\},B_1=\{(y,\beta_S):\parallel X_S(\beta_S- \hat{\beta}_S)\parallel ^2 \le  M \epsilon_n /4\},\\
& B_2=\{(y,\beta_S):\parallel X_S(\beta_S- \hat{\beta}_S)\parallel ^2 \le \parallel P_S X(\beta_{S+} -\beta^{\star})\parallel_2^2/2\}.
\end{split}
\end{equation}
\noindent
Hence on $Q^c_{n,S} \cap B_{M\epsilon_n}(S)$, we have 
\begin{equation}
\begin{split}
& [E_{\beta^{\star}}(b_0 + \frac{\alpha}{2} \Vert y-\hat{y}_S \Vert ^2)^{4/2}]^{1/4}\\
= &[b_0^2+\alpha \sigma_0^2 b_0(n-\vert S\vert+\frac{\parallel(I-P_S)X\beta^{\star}\parallel_2^2}{\sigma_0^2})+\frac{\alpha^2 \sigma^4_0}{4}(2(n-\vert S\vert)\\
& +\frac{4\parallel(I-P_S)X\beta^{\star}\parallel_2^2}{\sigma_0^2}+(n-\vert S\vert+\frac{\parallel(I-P_S)X\beta^{\star}\parallel_2^2}{\sigma_0^2})^2)]^{1/4}\\
\preceq & 2 n^{1/2}. \label{holder part 1}
\end{split}
\end{equation}
\noindent
Also by $1+x \le \exp(x)$,
\begin{equation}
\begin{split}
& [E_{\beta^{\star}} \frac{\{ b_0 + \frac{\alpha}{2} \Vert y-\hat{y}_S \Vert ^2\}^{(a_0+\alpha n/{2}-1/2)\cdot 2}}{\{ b_0 + \frac{\alpha}{2} \Vert y-\hat{y}_S \Vert ^2+ \frac{\gamma+\alpha}{2} \parallel X_S(\beta_S- \hat{\beta}_S)\parallel ^2 \}^{(a_0+\alpha n/{2}-1/2)\cdot 2}}]^{1/2} \\
\le & [E_{\beta^{\star}} (\frac{ (b_0 + \alpha \sigma_0^2 n)\mathbbm{1}_{A^c \cap B_1^c}}{ b_0 + \alpha \sigma_0^2 n+ \frac{\gamma+\alpha}{8} M \epsilon_n})^{(a_0+\alpha n/{2}-1/2)\cdot 2}+E_{\beta^{\star}}\mathbbm{1}_{A \cup B_1}]^{1/2}\\
\le & [\exp\{-\frac{\alpha+\gamma}{8 d_2}M\epsilon_n\}+\exp \{-n/10\}+\frac{2 c_2 C^\prime}{M \epsilon_n} \exp\{-\frac{M \epsilon_n}{64\sigma_0^2}\}]^{1/2}\\
\preceq & \exp \{-\frac{\alpha}{128 d_2}M\epsilon_n \} \label{holder part 2}
\end{split}
\end{equation}
\noindent
and
\begin{equation}
\begin{split}
& [E_{\beta^{\star}} \frac{1}{\{ b_0 + \frac{\alpha}{2} \Vert y-\hat{y}_S \Vert ^2+ \frac{\gamma+\alpha}{2} \parallel X_S(\beta_S- \hat{\beta}_S)\parallel ^2 \}^{(\vert S\vert/2+1/2)\cdot 4}}]^{1/4} \\
\le & [E_{\beta^{\star}} \frac{\mathbbm{1}_{B_2^c}}{\{ b_0 + \frac{\gamma+\alpha}{2} \parallel X_S(\beta_S- \hat{\beta}_S)\parallel ^2 \}^{(\vert S\vert/2+1/2)\cdot 4}}+ E_{\beta^{\star}} \mathbbm{1}_{B_2}]^{1/4}\\
\le & [E_{\beta^{\star}} \frac{1}{\{ b_0 + \frac{\gamma+\alpha}{4} \parallel P_S X(\beta_{S+} -\beta^{\star})\parallel_2^2 \}^{(\vert S\vert/2+1/2)\cdot 4}} \\
  &+ \frac{2 c_2 C^\prime}{M \epsilon_n} \exp\{-\frac{\parallel P_S X(\beta_{S+} -\beta^{\star})\parallel_2^2}{32\sigma_0^2}\}]^{1/4}\\
\preceq & \frac{1}{\{ b_0 + \frac{\gamma+\alpha}{4} \parallel P_S X(\beta_{S+} -\beta^{\star})\parallel_2^2 \}^{(\vert S\vert/2+1/2)}}+ C^\prime \exp\{-\frac{\parallel P_S X(\beta_{S+} -\beta^{\star})\parallel_2^2}{128\sigma_0^2} \} .\label{holder part 3}
\end{split}
\end{equation}

Then by(\ref{W_2 S part}), (\ref{holder part 1}), (\ref{holder part 2}), (\ref{holder part 3}) and $\parallel P_S X(\beta_{S+} -\beta^{\star})\parallel_2^2 = \parallel X_{S}(\beta_S-(X_{S}^T X_{S})^{-1}X_{S}^T X \beta^{\star})\parallel_2^2 $, we have
\begin{equation}
\begin{split}
&E_{\beta^{\star}}W_2 \\
\preceq &  \sum_{S^{\star} \nsubseteq S, \vert S\vert \le Cs^{\star}}\lbrace(\frac{\alpha+ \gamma}{2 \pi})^{\vert S\vert/2}\vert X_S^TX_S\vert^{1/2}\frac{\Gamma(a_0+\alpha n/2+\vert S\vert/2)}{\Gamma(a_0+\alpha n/2)}\\
& \times \int_{R^{\vert S\vert}}\frac{2 n^{1/2} \exp \{-\frac{\alpha}{128 d_2}M\epsilon_n\}}{\{ b_0 + \frac{\gamma+\alpha}{4} \parallel X_{S}(\beta_S-(X_{S}^T X_{S})^{-1}X_{S}^T X \beta^{\star})\parallel_2^2 \}^{(\vert S\vert/2+1/2)}}d\beta_S \}\\
& +   \sum_{S^{\star} \nsubseteq S, \vert S\vert \le Cs^{\star}}\lbrace(\frac{\alpha+ \gamma}{2 \pi})^{\vert S\vert/2}\vert X_S^TX_S\vert^{1/2}\frac{\Gamma(a_0+\alpha n/2+\vert S\vert/2)}{\Gamma(a_0+\alpha n/2)}\\
& \times \int_{R^{\vert S\vert}}2 C^\prime n^{1/2} \exp \{-\frac{\alpha}{128 d_2}M\epsilon_n\} \exp\{-\frac{\parallel X_{S}(\beta_S-(X_{S}^T X_{S})^{-1}X_{S}^T X \beta^{\star})\parallel_2^2}{128\sigma_0^2} \}d\beta_S \}\\
= &  \sum_{S^{\star} \nsubseteq S, \vert S\vert \le Cs^{\star}}\frac{\Gamma(a_0+\alpha n/2+\vert S\vert/2)\Gamma(1/2)}{\Gamma(a_0+\alpha n/2)\Gamma((\vert S\vert+1)/2)}2^{\vert S\vert/2+1}(\frac{n}{b_0})^{1/2} \exp\{-\frac{\alpha}{128 d_2}M\epsilon_n\}\\
&+\sum_{S^{\star} \nsubseteq S, \vert S\vert \le Cs^{\star}}2 C^\prime n^{1/2}\frac{\Gamma(a_0+\alpha n/2+\vert S\vert/2)}{\Gamma(a_0+\alpha n/2)}(\alpha+\gamma)^{\vert S\vert/2} \exp \{-\frac{\alpha}{128 d_2}M\epsilon\}
\\
\preceq & Cs^{\star}{{p} \choose {Cs^{\star}}}[n^{C s^{\star}/2+1/2}+(1+\gamma)^{\vert S\vert/2}n^{1+C s^{\star}/2}] \exp \{-\frac{\alpha}{128 d_2}M\epsilon_n\} \preceq   \exp({-G^{\prime}\epsilon_n})
\end{split}
\end{equation}
for some $G^{\prime}>0$, when $M > \frac{256C d_2}{\alpha}$.

For $W_3$,  by (\ref{beta_S given S}) and  Holder's inequality, we have 
\begin{displaymath}
\begin{split}
& E_{\beta^{\star}}\pi^n(B_{M\epsilon_n}(S)\vert S)\\
=&(\frac{\alpha+ \gamma}{2 \pi})^{\vert S\vert/2}\vert X_S^TX_S\vert^{1/2}\frac{\Gamma(a_0+\alpha n/2+\vert S\vert/2)}{\Gamma(a_0+\alpha n/2)}\\
& \times \int_{B_{M\epsilon_n}(S)}\frac{\{ b_0 + \frac{\alpha}{2} \Vert y-\hat{y}_S \Vert ^2\}^{(a_0+\alpha n/{2})}}{\{ b_0 + \frac{\alpha}{2} \Vert y-\hat{y}_S \Vert ^2+ \frac{\gamma+\alpha}{2} \parallel X_S(\beta_S- \hat{\beta}_S)\parallel ^2 \}^{(a_0+\alpha n/{2}+\vert S\vert/2)}} d\beta_S \\
\le &(\frac{\alpha+ \gamma}{2 \pi})^{\vert S\vert/2}\vert X_S^TX_S\vert^{1/2}\frac{\Gamma(a_0+\alpha n/2+\vert S\vert/2)}{\Gamma(a_0+\alpha n/2)}\\
& \times \int_{B_{M\epsilon_n}(S)} [E_{\beta^{\star}} \frac{\{ b_0 + \frac{\alpha}{2} \Vert y-\hat{y}_S \Vert ^2\}^{(a_0+\alpha n/{2}-1/2)\cdot 2}}{\{ b_0 + \frac{\alpha}{2} \Vert y-\hat{y}_S \Vert ^2+ \frac{\gamma+\alpha}{2} \parallel X_S(\beta_S- \hat{\beta}_S)\parallel ^2 \}^{(a_0+\alpha n/{2}-1/2)\cdot 2}}]^{1/2} \\
& \cdot [E_{\beta^{\star}} \frac{1}{\{ b_0 + \frac{\alpha}{2} \Vert y-\hat{y}_S \Vert ^2+ \frac{\gamma+\alpha}{2} \parallel X_S(\beta_S- \hat{\beta}_S)\parallel ^2 \}^{(\vert S\vert/2+1/2)\cdot 4}}]^{1/4} \\
& \cdot [E_{\beta^{\star}}(b_0 + \frac{\alpha}{2} \Vert y-\hat{y}_S \Vert ^2)^{4/2}]^{1/4}\cdot d\beta_S.
\end{split}
\end{displaymath}

When $S^{\star} \subseteq S$, 
\begin{displaymath}
\Vert y-\hat{y}_S \Vert ^2 \sim \sigma_0^2 \chi^2_{n-\vert S\vert}
\end{displaymath}
and 
\begin{displaymath}
\parallel X_S(\beta_S- \hat{\beta}_S)\parallel ^2 \vert \beta_S\sim \sigma_0^2 \chi^2_{\vert S\vert}(\frac{\parallel X(\beta_{S+} -\beta^{\star})\parallel_2^2}{\sigma_0^2})\quad ,
\end{displaymath}
Hence by Lemmas \ref{lemma1} and \ref{lemma2}, for ${\vert S\vert \le Cs^{\star}}$, on $ B_{M\epsilon_n}(S)$, we have 
\begin{displaymath}
\parallel P_S X(\beta_{S+} -\beta^{\star})\parallel_2^2 \ge M\epsilon_n  , \quad
P(\Vert y-\hat{y}_S \Vert ^2 > 2\sigma_0^2 n) \le \exp(-n/10) \quad ,
\end{displaymath}
and
\begin{displaymath}
\begin{split}
& P(\parallel X_S(\beta_S- \hat{\beta}_S)\parallel ^2 \le  M \epsilon_n /4\vert \beta_S)\\
\le &
P(\parallel X_S(\beta_S- \hat{\beta}_S)\parallel ^2 \le \parallel P_S X(\beta_{S+} -\beta^{\star})\parallel_2^2/2\vert \beta_S)\\
\le & \frac{C^\prime \sigma_0^2}{\parallel P_S X(\beta_{S+} -\beta^{\star})\parallel_2^2} 
\exp\{-\frac{\parallel P_S X(\beta_{S+} -\beta^{\star})\parallel_2^2}{32\sigma_0^2}\}\\
\le &\frac{2 c_2 C^\prime}{M \epsilon_n} \exp\{-\frac{\parallel P_S X(\beta_{S+} -\beta^{\star})\parallel_2^2}{32\sigma_0^2}\} \le \frac{2 c_2 C^\prime}{M \epsilon_n} \exp\{-\frac{M \epsilon_n}{64\sigma_0^2}\},
\end{split}
\end{displaymath}

Recalling the definitions of $A$, $B_1$ and $B_2$, 
by the same technique used for in $E_{\beta^{\star}}W_2$, we can prove that when $M > \frac{256 C d_2}{\alpha}$, $E_{\beta^{\star}}W_3 \preceq \exp({-G \epsilon_n})$, for some $G>0$.

We prove the Theorem by taking  $M \ge \max \{(2a+3) \frac{24 d_2}{\alpha},\frac{256C d_2}{\alpha}\}$.\\
\\
\\
\textbf{Proof of Theorem \ref{bernstein}:}\\
Since
\begin{displaymath}
\pi^n(\beta)= \sum_{S}\pi^n(\beta_S\vert S)\pi^n(S).
\end{displaymath}
Hence by convexity of $H^2$  and $H^2 \le 2$, we have
\begin{displaymath}
\begin{split}
&H^2(\Pi^n,\{N_{s^{\star}}(\hat{\beta}_{S^{\star}},\frac{\sigma_0^2}{\alpha+\gamma}(X^T_{S^{\star}}X_{S^{\star}})^{-1})\bigotimes \delta_{0S^{\star c}}\})\\
=& H^2(\sum_{S}\pi^n(\beta_S\vert S)\pi^n(S),N_{s^{\star}}(\hat{\beta}_{S^{\star}},\frac{\sigma_0^2}{\alpha+\gamma}(X^T_{S^{\star}}X_{S^{\star}})^{-1})\bigotimes \delta_{0S^{\star c}})  \\
\le & \sum_{S} \pi^n(S) H^2(\pi^n(\beta_S\vert S),N_{s^{\star}}(\hat{\beta}_{S^{\star}},\frac{\sigma_0^2}{\alpha+\gamma}(X^T_{S^{\star}}X_{S^{\star}})^{-1})\bigotimes \delta_{0S^{\star c}})\\
\le & \sum_{S \neq S^{\star}} 2\pi^n(S)+ H^2(\pi^n(\beta_{S^\star}\vert S^{\star}),N_{s^{\star}}(\hat{\beta}_{S^{\star}},\frac{\sigma_0^2}{\alpha+\gamma}(X^T_{S^{\star}}X_{S^{\star}})^{-1})\bigotimes \delta_{0S^{\star c}})\\
= & \sum_{S \neq S^{\star}} 2(1-\pi^n(S^{\star}))+ H^2(\pi^n(\beta_{S^\star}\vert S^{\star}),N_{s^{\star}}(\hat{\beta}_{S^{\star}},\frac{\sigma_0^2}{\alpha+\gamma}(X^T_{S^{\star}}X_{S^{\star}})^{-1})\bigotimes \delta_{0S^{\star c}}).
\end{split}
\end{displaymath}
We also have $E_{\beta^{\star}}\pi^n(S^{\star})\to 1$ by dominated cobvergence theorem. So it suffices to show that 
\begin{displaymath}
E_{\beta^{\star}} H^2(\pi^n(\beta_{S^\star}\vert S^{\star}),N_{s^{\star}}(\hat{\beta}_{S^{\star}},\frac{\sigma_0^2}{\alpha+\gamma}(X^T_{S^{\star}}X_{S^{\star}})^{-1})\bigotimes \delta_{0S^{\star c}}) \to 0.
\end{displaymath}
This we prove by showing   expectation of the Hellinger affinity
\begin{equation}
E_{\beta^{\star}}\int_{R^{s^{\star}}} \sqrt{\pi^n({\beta}_{S^{\star}}\vert S^{\star}) \cdot N_{s^{\star}}({\beta}_{S^{\star}}\vert\hat{\beta}_{S^{\star}} , \frac{\sigma_0^2}{\alpha+\gamma}(X^T_{S^{\star}}X_{S^{\star}})^{-1})}  d\beta_{S^{\star}} \to 1.
\end{equation}

To this end, let  
\begin{displaymath}
Q=\{y: (n-s^{\star})- \sqrt{n-s^{\star}} \log(n-s^{\star}) \le \Vert y-\hat{y}_{S^{\star}} \Vert ^2/{\sigma_0^2}\le (n-s^{\star})+ \sqrt{n-s^{\star}} \log(n-s^{\star})\}.
\end{displaymath}
Then by $1+x \le \exp(x)$,
\begin{equation}
\begin{split}
& E_{\beta^{\star}}\int_{R^{s^{\star}}} \sqrt{\pi^n({\beta}_{S^{\star}}\vert S^{\star}) \cdot N_{s^{\star}}({\beta}_{S^{\star}}\vert\hat{\beta}_{S^{\star}} , \frac{\sigma_0^2}{\alpha+\gamma}(X^T_{S^{\star}}X_{S^{\star}})^{-1})}  d\beta_{S^{\star}}\\
\ge & {(\sigma_0^2)}^{-{s^{\star}}/4}(\frac{\alpha+ \gamma}{2 \pi})^{s^{\star}/2}\vert X_{S^{\star}}^TX_{S^{\star}}\vert^{1/2}\frac{\Gamma^{\frac{1}{2}}(a_0+\alpha n/2+{s^{\star}}/2)}{\Gamma^{\frac{1}{2}}(a_0+\alpha n/2)}\\
& \times E_{\beta^{\star}}\int_{R^{s^{\star}}} \mathbbm{1}_{Q} \cdot (b_0 + \frac{\alpha}{2} \Vert y-\hat{y}_{S^{\star}} \Vert ^2)^{-s^{\star}/4} \cdot \exp\{{-\frac{{(\alpha+\gamma)} \parallel X_{S^{\star}}(\beta_{S^{\star}}- \hat{\beta}_{S^{\star}})\parallel ^2}{4\sigma_0^2}}\}\\
& \cdot [1+\frac{{(\alpha+\gamma)} \parallel X_{S^{\star}}(\beta_{S^{\star}}- \hat{\beta}_{S^{\star}})\parallel ^2}{2b_0 + {\alpha} \Vert y-\hat{y}_{S^{\star}} \Vert ^2}]^{-(a_0+\alpha n/{2}+{s^{\star}}/2)/2}d\beta_{S^{\star}}\\
\ge & {(\sigma_0^2)}^{-{s^{\star}}/4}(\frac{\alpha+ \gamma}{2 \pi})^{s^{\star}/2}\vert X_{S^{\star}}^TX_{S^{\star}}\vert^{1/2}\frac{\Gamma^{\frac{1}{2}}(a_0+\alpha n/2+{s^{\star}}/2)}{\Gamma^{\frac{1}{2}}(a_0+\alpha n/2)}\\
& \times E_{\beta^{\star}}\int_{R^{s^{\star}}} [b_0 + \frac{\alpha \sigma_0^2}{2} ((n-s^{\star})+ \sqrt{n-s^{\star}} \log(n-s^{\star}) )]^{-s^{\star}/4} \\
& \cdot \exp\{{-\frac{{(\alpha+\gamma)} \parallel X_{S^{\star}}(\beta_{S^{\star}}- \hat{\beta}_{S^{\star}})\parallel ^2}{4\sigma_0^2}}\}\\
& \cdot \exp\{ - \frac{{(\alpha+\gamma)} \parallel X_{S^{\star}}(\beta_{S^{\star}}- \hat{\beta}_{S^{\star}})\parallel ^2}{4b_0 + 2{\alpha} \sigma_0^2 ((n-s^{\star})- \sqrt{n-s^{\star}} \log(n-s^{\star}))}(a_0+\alpha n/{2}+{s^{\star}}/2) \}d\beta_{S^{\star}}\\
\succeq & {(\sigma_0^2)}^{-{s^{\star}}/4}(\frac{\alpha+ \gamma}{2 \pi})^{s^{\star}/2}\vert X_{S^{\star}}^TX_{S^{\star}}\vert^{1/2}\frac{\Gamma^{\frac{1}{2}}(a_0+\alpha n/2+{s^{\star}}/2)}{\Gamma^{\frac{1}{2}}(a_0+\alpha n/2)} \\
& \cdot [\alpha \sigma^2_0 n /2+\alpha \sigma_0^2\sqrt{n} \log n /2]^{-{s^{\star}}/4} \\
& \cdot E_{\beta^{\star}}\int_{R^{s^{\star}}} \exp\{{-(\frac{1}{2}+\frac{\alpha n/2+\sqrt{n}/(2 \log n)}{\alpha(n-2\sqrt{n} \log n)})\frac{{(\alpha+\gamma)} \parallel X_{S^{\star}}(\beta_{S^{\star}}- \hat{\beta}_{S^{\star}})\parallel ^2}{2\sigma_0^2}}\}  d\beta_{S^{\star}}\\
= & \frac{\Gamma^{\frac{1}{2}}(a_0+\alpha n/2+{s^{\star}}/2)}{\Gamma^{\frac{1}{2}}(a_0+\alpha n/2)(\alpha n/2)^{{s^{\star}}/4}}[1+\frac{ \log n}{\alpha \sqrt{n}} ]^{-{s^{\star}}/4}(1+\frac{\sqrt{n}/(2 \log n)+\alpha \sqrt{n} \log n}{\alpha (n-2 \sqrt{n} \log n)})^{-{s^{\star}}/2} \\
\succeq & \sqrt{\frac{a_0+\alpha n /2}{\alpha n /2} \cdot \frac{a_0+\alpha n /2 +1}{\alpha n /2} \cdots \frac{a_0+\alpha n /2 +s^{\star}-1}{\alpha n /2}}  \\
\geq & (\frac{a_0+\alpha n /2}{\alpha n /2} )^{s^ \star /4} \to  1.
\end{split}
\end{equation}
\noindent
\\
\\
\textbf{Proof of Corollary \ref{uncertainty quantification}: }\\
Let $\psi = x^T {\beta}$. Similar to the proof of Theorem \ref{bernstein}  we have 
\begin{displaymath}
E_{\beta^{\star}}d_{TV}(\Pi^n_{\psi},N(x^T_{S^{\star}}\hat{\beta}_{S^{\star}},\frac{\sigma_0^2}{\alpha+\gamma}x^T_{S^{\star}}(X^T_{S^{\star}}X_{S^{\star}})^{-1}x_{S^{\star}})) \to 0 ,
\end{displaymath}
where $\Pi^n_{\psi}$ is the derived posterior distribution of  $\psi = x^T {\beta}$.  
Hence $\frac{x^T_{S^{\star}}\hat{\beta}_{S^{\star}}-t_{\gamma}}{\sqrt{\frac{\sigma_0^2}{\alpha+\gamma}x^T_{S^{\star}}(X^T_{S^{\star}}X_{S^{\star}})^{-1}x_{S^{\star}}}}$   is asymptotically equals $\Phi^{-1}(\gamma)$. Then writing $Z$ as a $N(0,1)$ variable,
\begin{displaymath}
\begin{split}
P_{\beta^{\star}}(t_{\gamma} \ge x^T {\beta}^{\star})=& P_{\beta^{\star}}(\frac{x^T_{S^{\star}}\hat{\beta}_{S^{\star}}-x^T {\beta}^{\star}}{\sqrt{\frac{\sigma_0^2}{\alpha+\gamma}x^T_{S^{\star}}(X^T_{S^{\star}}X_{S^{\star}})^{-1}x_{S^{\star}}}} \ge \frac{x^T_{S^{\star}}\hat{\beta}_{S^{\star}}-t_{\gamma}}{\sqrt{\frac{\sigma_0^2}{\alpha+\gamma}x^T_{S^{\star}}(X^T_{S^{\star}}X_{S^{\star}})^{-1}x_{S^{\star}}}})\\
= & P_{\beta^{\star}}(Z \ge \frac{x^T_{S^{\star}}\hat{\beta}_{S^{\star}}-t_{\gamma}}{\sqrt{\frac{\sigma_0^2}{\alpha+\gamma}x^T_{S^{\star}}(X^T_{S^{\star}}X_{S^{\star}})^{-1}x_{S^{\star}}}}) \ge  1- \gamma + o(1), \qquad n \to \infty.
\end{split}
\end{displaymath}
\noindent

\section{Final Remarks}
The paper extends the work of Ryan Martin and his colleagues who proposed empirical prior for linear regression models. The contribution of this article is to extend their work for unknown error variance. The theoretical advancement is the derivation of new results related to model selection consistency , posterior contraction rates as well as  a new Bernstein von-Mises theorem in our framework. An important open question is whether similar results can be found in the set up of  \citep{Castillo:2015} who proposed Laplace priors for a linear regression problem with known error variance.




\end{document}